\documentclass[12pt]{amsart}

\newcounter{defcounter}
\usepackage{amssymb,amsmath,amscd,graphicx, enumitem,
latexsym,amsthm}
\usepackage{amssymb,latexsym,eufrak,amsmath,amscd,graphicx}
  \usepackage[all]{xy}
  \usepackage{pdfsync}
\theoremstyle{plain}
\newtheorem{theorem}{Theorem}
\newtheorem{proposition}[theorem]{Proposition}
\newtheorem{corollary}[theorem]{Corollary}

\newtheorem{proposition.definition}[theorem]{Proposition/Definition}

\newtheorem{theoremalpha}{Theorem}

\newtheorem{propositionalpha}[theoremalpha]{Proposition}

\theoremstyle{definition}

\newtheorem{remark}[theorem]{Remark}

\newcommand{\lra}{\longrightarrow}

\newcommand{\noi}{\noindent}
\newcommand{\PP}{\mathbf{P}}

\newcommand{\ZZ}{\mathbf{Z}}

\newcommand{\OO}{\mathcal{O}}

\newcommand{\fra}{\mathfrak{a}}

\newcommand{\HH}[3]{H^{{#1}} \big( {#2} , {#3}
\big) }

\newcommand{\hh}[3]{h^{{#1}} \big( {#2} , {#3}
\big) }

\newcommand{\pr}{\prime}

\newcommand{\lin}{\equiv_{\text{lin}}}

\newcommand{\dra}{\dashrightarrow}

\newcommand{\Linser}[1]{| \mspace{1.5mu} {#1}
\mspace{1.5mu} |}
\newcommand{\linser}[1]{\Linser{  {#1}  }}

\newcommand{\Konno}{\textnormal{Konno}}

\numberwithin{theorem}{section}
\numberwithin{equation}{section}

\begin{document}

\title
{The Konno invariant of some algebraic varieties}

 \author{Lawrence Ein}
  \address{Department of Mathematics, University of Illinois at Chicago, 851 South Morgan St., Chicago, IL  60607}
 \email{{\tt ein@uic.edu}}
 \thanks{Research of the first author partially supported by NSF grant DMS-1801870.}
 
 \author{Robert Lazarsfeld}
  \address{Department of Mathematics, Stony Brook University, Stony Brook, New York 11794}
 \email{{\tt robert.lazarsfeld@stonybrook.edu}}
 \thanks{Research of the  second author partially supported by NSF grant DMS-1739285.}

\maketitle

 \section*{Introduction}

Let $X$ be a smooth complex projective variety of dimension $n \ge 2$. We define the \textit{Konno invariant} of $X$ to be the minimal geometric genus of 
a pencil of connected divisors on $X$:
\[
\Konno(X) \ = \ \min \Big\{ \, g \, \Big | \ 
  \begin{minipage}{3in}{ $\exists$ a connected rational pencil 
\ $ \pi : X \dashrightarrow \PP^1 $ \ whose general fibre $F$ has $p_g(F) = g$ }\end{minipage}
\ \Big \}.\]
(The geometric genus of an  irreducible projective variety is understood to be the $p_g$ of any desingularization.) This invariant was introduced and studied by Konno \cite{Konno}, who computed it for smooth surfaces in $\PP^3$: he proves that in this case  pencils of minimal genus are given by projection from a line. In general, one should view $\Konno(X)$ as one of many possible measures of the ``complexity" of $X$. The purpose of this note is to estimate this invariant for some natural classes of varieties.

Our first result involves the Konno invariant of varieties such as general complete intersections whose Picard groups are generated by a very ample divisor.
\begin{propositionalpha} \label{ThmA}
Assume that $\textnormal{Pic}(X)= \ZZ \cdot [H]$ where $H$ is a very ample divisor on $X$. Then
\[
h^0(K_X) \, - \, h^0(K_X - H) \ \le \ \Konno(X) \ \le \ h^0(K_X + H) - h^0(K_X) \, + \, h^1(K_X).
\]
\end{propositionalpha} 
\noi So for example if
\[    X_d \ \subseteq \ \PP^{n+1}\]
is a hypersurface of  degree $d$ (very general if $n = 2$), then as a function of $d$
 \[ \Konno(X_d) \sim \frac{d^n}{n!}. \]
Observe that at least when $H^1(K_X) = 0$, the upper bound in Proposition A is the geometric genus of a general pencil in $\linser{H}$. However if $h^0(X, H) \gg n$, then one can construct special pencils of highly singular hypersurfaces of somewhat smaller geometric genus.

Our second result deals with polarized K3 surfaces of large degree.
\begin{theoremalpha} \label{K3ThmIntro}
Let $(S_d, L_d)$ be a polarized K3 surface of genus $d\ge 3$, and assume that 
\[ \textnormal{Pic}(S_d) \ = \ \ZZ \cdot [L_d]. \]
Then
\[ \Konno(S_d) \, \in \, \Theta\big(\sqrt{d}\big), \]
i.e. there are constants $C_1, C_2> 0$ such that
\[ C_1 \cdot \sqrt{d}\ < \ \Konno(S_d) \ < \ C_2 \cdot \sqrt{d} \]
for all such surfaces $S_d$ and all large $d$. 
\end{theoremalpha}
\noi It is conjectured by Stapleton \cite{Stapleton} that the same statement holds for the degree of irrationality $\textnormal{irr}(S_d)$ of $S_d$, but this remains an intriguing open problem. An analogue of Theorem B is also valid for general polarized abelian surfaces.

The proof of Proposition \ref{ThmA} occupies \S 1. It arises as a special case of a somewhat more general (but very elementary) result dealing with one-dimensional families of hypersurfaces. Section 2 is devoted to a more refined lower bound for surfaces, from which we deduce Theorem \ref{K3ThmIntro};  following  \cite{Konno}, the key point here is to use some classical statements of Noether computing the invariants of a linear series in terms of its multiplicities at finite and infinitely near points. We conclude with an appendix in which we review Noether's formulae, and show in particular how they lead to quick proofs of theorems of Deligne-Hoskin and Lech concerning finite colength ideals on a surface.

We are grateful to Francesco Bastianelli, Craig Huneke, David Stapleton  and Ruijie Yang for useful discussions. 

\section{Geometric genera of covering families of divisors}

Let $X$ be a smooth complex projective variety of dimension $n$. 

\begin{theorem} \label{Sect1Thm}
Let $\{ F_t\}_{t \in T}$ be a family of divisors on $X$ parametrized by a smooth curve $T$. Assume that the $F_t$ are generically irreducible and that they cover $X$,  and denote by $F \subseteq X$ a general element in the family.
Then
\begin{equation}  \label{Thmeqn} p_g(F) \ \ge \ \hh{0}{X}{\OO_X( K_X)} \, - \, \hh{0}{X}{\OO_X( K_X - F)}.\end{equation}
\end{theorem}
\noi Note that although we don't assume that the $\{F_t\}$ are all linearly equivalent, the expression on the right is independent of the choice of a generic element of the family. Observe also that it can happen that equality holds in \eqref{Thmeqn}: for example one can take $X = C \times F$ where $C$ is an elliptic curve.
\begin{proof} We adapt the elementary argument proving \cite[Theorem 1.10]{BDELU}. One can construct a diagram:
\[
\xymatrix{Y\ar[d]_\pi \ar[r]^\mu &X\\
T}
\]
where $Y$ is smooth, and almost all fibres 
\[  E_t \, =_\text{def} \, \pi^{-1}(t) \ \subseteq \ Y \]
are smooth irreducible divisors mapping birationally to their images $F_t \subseteq X$. Denote by $E$ a general fibre of $\pi$, with $F = \mu(E) \subseteq X$. So by definition $p_g(F)  =   \hh{0}{E}{\OO_E(K_E)}$. Now 
\[ K_Y \ \lin \ \mu^* K_X \, + \, R, \]
where $R$ is effective, and $K_E = K_Y \mid E$. Therefore
\[ p_g(F) \ \ge \ \hh{0}{E}{\mu^*\OO_X(K_X) \mid E}. \] On the other hand, $\mu^*$ gives rise to a natural injection
\[ \HH{0}{F}{\OO_F(K_X)} \hookrightarrow \HH{0}{E}{\OO_E(\mu^*
 K_X)},\]
so we arrive finally at the inequality
\[ p_g(F) \ \ge \ \hh{0}{F}{\OO_X(K_X)\mid F}. \]
The statement then follows by using the exact sequence 
\[ 0 \lra \OO_X(K_X - F) \lra \OO_X(K_X) \lra \OO_F(K_X) \lra 0
\] to estimate $\hh{0}{F}{\OO_F(K_X)}$.
\end{proof}

\begin{proof}[Proof of Proposition \ref{ThmA}]
We apply the previous result with $F \in \linser{rH}$ for some $r \ge 1$. The right hand side of \eqref{Thmeqn} is mimimized when $r = 1$, and the lower bound follows. The upper bound follows by considering a general pencil in $\linser{H}$.
\end{proof}

\begin{remark} \textbf{(Covering families of curves)}.  By a similar argument, if $\{ C_t \}_{t\in T}$ is a family of irreducible curves of geometric genus $g$ that covers a Zariski-open subset of $X$, then
\begin{equation}  (2g  - 2) \ \ge \ \big ( K_X \cdot C \big ), \label{Cov.Fam.Curves} \end{equation}
where $C$ is a general curve in the family. 
\end{remark}

\section{The Konno invariant of an algebraic surface}

The inequality of Theorem \ref{Sect1Thm} says nothing for varieties with trivial canonical bundle. In the case of surfaces we prove here a variant that does yield non-trivial information in this case. The approach is inspired by the arguments of Konno in \cite{Konno}.

\begin{theorem} \label{SfBoundThm} Let $S$ be a smooth complex projective surface, and let $L$ be an ample line bundle on $S$. Fix a two-dimensional subspace $V \subseteq \HH{0}{S}{L}$ with only isolated base-points defining a rational pencil
\[ \phi_{|V|} : S \dra \PP^1\]
with generically irreducible fibres. If $g$ denotes the geometric genus of the general fibre, then
\begin{equation}  \left ( 2g-2 \right)  \ge \ \big( K_S \cdot L \big) \, + \, \sqrt{(L^2)}\end{equation} \label{Thm2.1Ineq} 
\end{theorem}
\begin{remark} Compare the bound appearing above in equation \eqref{Cov.Fam.Curves}. \end{remark} 
\begin{proof}
By a sequence of blowings up at points, we construct a resolution of the indeterminacies of $\linser{V}$:
\[
\xymatrix{S^\pr\ar[d]_\pi \ar[r]^\mu &S\\
\PP^1 }.
\]
We can suppose that the centers of the blowings-up are the (actual and infinitely near) base-points of $\linser{V}$. Let $m_i$ denote the multiplicity of the proper transform of a general curve $C \in \linser{V}$ at the $i^{\text{th}}$ base-point, and denote by $C^\pr$ proper transform of $C$ in $S^\pr$, so that $C^\pr$ is a general fibre of $\pi$ and $g = g(C^\pr)$. Then by a classical theorem of Noether, which we recall in the Appendix  (Proposition \ref{Noether.Prop}), one has
\begin{align*}
\big( C \cdot C \big)_S\ = \ (C^\pr \cdot C^\pr)_{S^\pr}+ \sum m_i^2,
\end{align*}i.e.
\[ \big(L^2 \big) \ = \ \sum m_i^2 \ . \tag{*} \]
Furthermore,
\[ \left(2 p_a(C) - 2\right) \ = \ (2 g - 2) \, + \, \sum m_i(m_i-1). \]
But $ (2p_a(C) - 2)=  (K_X+ L) \cdot L$,  so we find that
\[  \left( 2g - 2 \right) \ = \ \big( L \cdot K_X \big) \, + \, \sum m_i . \]
The stated inequality  \eqref{Thm2.1Ineq} then follows from (*) and the fact that $\sum x_i \ge \sqrt{\sum x_i^2}$ for any non-negative real numbers $x_i$.
\end{proof}

\begin{proof}[Proof of Theorem \ref{K3ThmIntro}]
Let $(S,L) = (S_d, L_d)$ be a polarized K3 surface of genus $d$, so that 
\[
\big( L^2 \big) \ = \ 2d - 2 \ \ , \ \ \hh{0}{S}{L} \ = \ d +1.
\] We assume that Pic$(S) = \ZZ \cdot [L]$, and consider a rational pencil 
\[ \phi : S \dra \PP^1 \]
of curves of geometric genus $g$. Then for some $r \ge 1$, $\phi$ is defined by a two-dimensional subspace $V \subseteq \HH{0}{S}{rL}$ with isolated base points. Theorem \ref{SfBoundThm} implies that 
\[ 2g - 2 \ \ge \  r \cdot \sqrt{2d - 2} \ \ge \ \sqrt{2d -2}, \]
so $\Konno(S) \ge C_1 \cdot \sqrt{d} $ for suitable $C_1 > 0$. 

It remains to construct a pencil of small genus, for which we use an argument of Stapleton \cite{Stapleton}. Specifically, fix a point $x \in S$, and choose an integer $m \ge 1 $ so that 
\[ (m+2)^2 \ \ge \ 2d \ \ge \ (m+1)^2 . \tag{*}\]
It follows from (*) that 
\[ (d+1) \, - \, \frac{m(m+1)}{2} \ \ge \ 2, \]
and therefore \[\hh{0}{S}{L \otimes I_x^m} \ge 2 \]
All the curves in $\linser{L \otimes I_x^m}$ are reduced and irreducible, so we get a pencil of curves of geometric genus $g$ with
\[ (2g - 2) \ \le \ (2d-2) \, - \, m(m-1). \] But $2d - m^2 \le 4m + 4$ thanks to (*), and one then finds that $(2g - 2) \le 3 \cdot \sqrt{2d}$. Thus we have constructed a pencil of geometric genus $\le C_2 \cdot \sqrt{d}$ for suitable $C_2$, as required. 
\end{proof}

\begin{remark} (\textbf{Abelian surfaces}). Let $A$ be an abelian surface with a polarization of type $(1,d)$ that generates the N\'eron-Severi group of $A$. Then essentially the same argument shows that 
\[ \Konno(A) \ \in \ \Theta(d).\]

\end{remark}

\begin{remark} (\textbf{Non-linear families of curves on K3 surfaces)}.
One can view Theorem \ref{K3ThmIntro} as asserting there are no \textit{lines}  
\[ \PP^1 \, \subseteq \, \linser{L_d} \]
contained in the locus of curves having small geometric genus. It would be interesting to know whether one can also rule out the presence of rational curves of higher degree. For example, a general polarized K3 surface $(S, L_d)$ contains a (non-compact) two-dimensional family of nodal curves of geometric genus $p_g = 2$. Does this surface contain any rational curves? More generally, do  the Severi varieties parametrizing nodal curves of small geometric genus in $\linser{L_d}$ exhibit hyperbolic tendencies?
\end{remark}

\begin{remark} (\textbf{Calabi-Yau or hyper-K\"ahler manifolds}). Can one establish non-trivial lower bounds on the Konno invariant of a Calabi-Yau or hyper-K\"ahler manifold?
\end{remark}

\appendix
\section{Noether's formulas for linear series on surfaces}

We quickly review Noether's classical approach to invariants of linear series on surfaces, upon which the proof of Theorem \ref{SfBoundThm} was based. 
Besides accomodating the convenience of the reader, our motivation is to show how these ideas lead to quick proofs of results of Deligne-Hoskin and Lech.

Let $S$ be a smooth projective surface, $L$ a line bundle on $S$, and $V\subseteq \HH{0}{S}{L}$ a vector space of dimension $\ge 2$ defining a linear series with only isolated base-points. Given a point $x \in S$, the \textit{multiplicity} or \textit{order of vanishing} $m$ of $\linser{V}$ at $x$ is the multiplicity at $x$ of a general curve $C \in \linser{V}$. Equivalently, if $\mu_1= \textnormal{bl}_x : S_1 \lra S$ is the blowing-up of $S$ at $x$ with exceptional divisor $E$, $m$ is the unique integer such that 
\[ V_1 \, =_\text{def} \, \mu_1^*(V) (-mE) \ \subseteq \ \HH{0}{S_1}{\mu^*L \otimes \OO_{S_1}(-mE)}\]
again has at most isolated base-points. We call $\linser{V_1}$ the \textit{proper transform} of $\linser{V}$ and $L_1 =_{\text{def}}{\mu^*L \otimes \OO_{S_1}(-mE)} $ the \textit{proper transform} of $L$ on $S_1$.

Noether's result is the following:
\begin{proposition} \label{Noether.Prop}
Given $V \subseteq \HH{0}{S}{L}$ as above, let
\[ \mu : S^\pr \lra S \]
be a log resolution of $\linser{V}$ constructed as a sequence of blowings-up at points, so that the proper transform $V^\pr$ of $V$ on $S^\pr$ is base-point free. Denote by $m_i$ the multiplicity of the proper transform of $\linser{V}$ at the center of the $i^{\text{th}}$ blow-up. Then:
\begin{itemize}
\item[$(i)$.] Writing $L^\pr$ for the proper transform of $L$ on $S^\pr$, one has \[ \big (L^\pr \cdot L^\pr \big)_{S^\pr} \ = \ \big( L \cdot L \big)_S \, - \, \sum m_i^2. \]
\item[$(ii)$.] Let $C \in \linser{V}$ be a general curve, and let $C^\pr \in \linser{V^\pr}$ be its proper transform on $S^\pr$. Then $C^\pr$ is smooth, and 
\[ (2g(C^\pr)-2) \ = \ (2p_a(C)-2) - \sum m_i(m_i-1).\]
\end{itemize}
\end{proposition}
\begin{proof}
Let $\widetilde{E_i}$ be the \textit{total transform} on $S^\pr$ of the exceptional divisor created at the $i^{th}$ blow-up. Then
\[   \big(\widetilde E_i \cdot \widetilde E_j) \ = \ \begin{cases} -1  \ &\text{ if } \ i \, = \, j \\ \ 0 &\text{ if } i \, \ne \, j \end{cases}.\]
Moreover, \[
\big( \widetilde{E}_i \cdot \mu^* B \big) \  = \ 0 \]
for any line bundle $B$ on $S$.
On the other hand, by definition of the $m_i$:
\[ L^\pr \ = \ \mu^*L \otimes \OO_{S^\pr}(-\sum m_i \widetilde E_i), \]
and (i) follows. For (ii), note that 
\begin{equation} K_{S^\pr}\ \lin\ \mu^* K_S \, + \, \sum \widetilde E_i, \label{Rel.K.Eqn} \end{equation}
and apply the adjunction formula. The smoothness of $C^\pr$ follows from the fact that it is a general member of a base-point free linear system.
\end{proof}

We next show how these ideas lead to a very quick proof of a formula of Deligne \cite[Thm. 2.13]{Deligne} and Hoskin \cite{Hoskin}.
\begin{proposition}  Let 
\[ \fra \ \subseteq \ \OO_S \]
be an integrally closed ideal of finite colength cosupported at a point $x \in S$, and denote by $m_i$ the orders of vanishing of $\fra$ at $x$ and all infinitely near base-points of $\fra$. Then
\[
\textnormal{colength}(\fra) \ = \ \frac{\sum m_i(m_i+1)}{2}. 
\]
\end{proposition}
\begin{proof} From the exact sequence
\[ 0 \lra \fra \lra \OO_S \lra \OO_S/\fra \lra 0 \]
we see that \[\textnormal{colength} (\fra) \ = \ \chi(S, \OO_S) \, - \, \chi(S, \fra). \]
Now pass to a log resolution $\mu : S^\pr \lra S$ of $\fra$ so that
\[  \fra \cdot \OO_{S^\pr} \ = \ \OO_{S^\pr}(-A) \ \ , \ \ \text{with} \ A = \sum m_i \widetilde E_i.  \]
Then $\OO_{S^\pr}(-A)$ is globally generated with respect to $\mu$, so by a theorem of Lipman \cite[Theorem 12.1]{Lipman} $R^1\mu_*\OO_{S^\pr}(-A) = 0$.\footnote{In our setting, the vanishing in question is very elementary. In fact, the question being local, one can replace $S$ by an affine neighborhood of $x$, so that $\OO_{S^\pr}(-A)$ is globally generated. Choose a general section $s \in \Gamma\big ( \OO_{S^\pr}(-A)\big)$ cutting out a curve $\Gamma \subseteq S^\pr$. Then $\Gamma$ is finite over $S$, so the vanishing of $R^1\mu_*\OO_{S^\pr}(-A)$ follows from the exact sequence $0 \lra \OO_{S^\pr} \lra \OO_{S^\pr}(-A) \lra \OO_\Gamma(-A) \lra 0$. } Moreover $\fra = \mu_* \OO_{S^{\pr}}(-A)$ thanks to the integral closure of $\fra$, and hence
\[ \chi(S, \fra) \ = \ \chi(S^\pr, \OO_{S^\pr}(-A)).\] The statement then follows by using \eqref{Rel.K.Eqn} and Riemann-Roch to calculate:
\begin{align*}
\textnormal{colength}(\fra) \ &= \ \chi(S^\pr, \OO_{S^\pr}) \, - \, \chi(S^\pr, \OO_{S^\pr}(-A)) \\ &= \ \chi(S^\pr, \OO_{S^\pr}) \, - \, \left( \frac{ \big( -A \cdot (-A - K_{S^\pr}) \big)}{2} + \chi(S^\pr, \OO_{S^\pr})\right) \\ &= - \,\frac{ \Big( (-\sum m_i \widetilde E_i) \cdot \big( -\
\mu^* K_S -\sum (m_i +1) \widetilde E_i\big)\Big)}{2}\\ &= \ \frac{\sum m_i(m_i+1)}{2},
\end{align*}
as required.
\end{proof}

Finally, we note that the Proposition implies the two-dimensional case of an inequality of Lech \cite{Lech}.
\begin{corollary} \label{Lech}
Let $\fra \subseteq \OO_S$ be an ideal of finite colength. Then
\[ e(\fra) \, + \, e(\fra)^{\tfrac{1}{2}} \  \le \ 2 \cdot \textnormal{colength}(\fra), \]
where $e(\fra)$ denotes the Samuel multiplicity of $\fra$.  In particuar, 
\[ e(\fra)   \  \le \ 2 \cdot \textnormal{colength}(\fra). \]
\end{corollary} 
\begin{remark} The first inequality is the  two-dimensional smooth case of \cite[(1.1)]{Huneke}.
\end{remark}
\begin{proof}
We may assume that $\fra$ is cosupported at a single point. Furthermore, if $\overline{\fra} \subset \OO_X$ denotes the integral closure of $\fra$, then
\[  e(\overline{\fra}) \, = \, e(\fra) \ \ \text{ and } \ \ \textnormal{colength}(\overline{\fra}) \, \le \, \textnormal{colength}(\fra). \]
Thus we may assume in addition that $\fra$ is integrally closed, putting us in the setting of the previous result. Keeping notation as in the proof of that statement, one has
\begin{align*}
e(\fra) \ = \  - \, \big( A \cdot A\big) \  &= \ \sum m_i^2 \\
\textnormal{colength}(\fra) \ &= \ \frac{\sum m_i(m_i+1)}{2}.
\end{align*}
Recalling again that $\sum m_i \ge \sqrt{\sum m_i^2}$, the required inequality follows.
\end{proof}

 %
 %
 %
 %


\begin{thebibliography}{EMS}
 \setlength{\parskip}{3pt}

\bibitem{BDELU} Francesco Bastianelli, Pietro De Poi, Lawrence Ein, Robert Lazarsfeld and Brooke Ullery, Measures of irrationality for hyersurfaces of large degree, \textit{Compos. Math.} \textbf{153} (2017), 2368 -- 2393.
 
 \bibitem{Deligne} Pierre Deligne, Intersections sur les surfaces r\'eguli\`eres, in SGA7 II, LNM 340, 1973.
 
  
 \bibitem{Hoskin} Michael Hoskin, Zero-dimensional valuation ideals associated with plane curve branches, \textit{Proc. London Math. Soc.} \textbf{6} (1956) 70--99.

\bibitem{Huneke} Craig Huneke, Ilya Smirnov and Javid Validashti, A generalization of an inequality of Lech relating multiplicity and colength, preprint, \texttt{arXiv:1711.06951}.


\bibitem{Konno} Kazuhiro Konno, Minimal pencils on smooth surfaces in $\PP^3$, \textit{Osaka J. Math} \textbf{45} (2008), 789 -- 805. 

\bibitem{Lech} C. Lech, Note on multiplicities of ideals, \textit{Ark. Mat.} \textbf{4} (1960), 63--86.

\bibitem{Lipman} Joseph Lipman, Rational singularities, with applications to algebraic surfaces and unique factorization, \textit{IHES Publ. Math.} \textbf{36} (1969), 195--279.

\bibitem{Stapleton} David Stapleton, Stony Brook Thesis, 2017.




 \end{thebibliography}
 \end{document}